\documentclass{scrartcl}
\usepackage{amssymb,amsmath,amsthm}

\usepackage[utf8]{inputenc}
\usepackage[T1]{fontenc}

\allowdisplaybreaks

\usepackage{newunicodechar}
\newunicodechar{α}{\alpha}
\newunicodechar{β}{\beta}
\newunicodechar{γ}{\gamma}
\newunicodechar{δ}{\delta}
\newunicodechar{ε}{\varepsilon}
\newunicodechar{η}{\eta}
\newunicodechar{∞}{\infty}
\newunicodechar{ℝ}{\mathbb{R}}
\newunicodechar{ℕ}{\mathbb{N}}
\newunicodechar{Δ}{\Delta}
\newunicodechar{φ}{\varphi}
\newunicodechar{ψ}{\psi}
\newunicodechar{χ}{\chi}
\newunicodechar{λ}{\lambda}
\newunicodechar{∂}{\partial}
\newunicodechar{μ}{\mu}
\newunicodechar{ν}{\nu}
\newunicodechar{ξ}{\xi}
\newunicodechar{ϕ}{\phi}

\newcommand{\f}[2]{\frac{#1}{#2}}
\newcommand{\Bbar}{\overline{B}}
\newcommand{\set}[1]{\left\{#1\right\}}
\newcommand{\Ombar}{\overline{\Omega}}
\newcommand{\Omebar}{\overline{\Omega_{ε}}}
\newcommand{\Om}{\Omega}
\newcommand{\Ome}{\Om_{ε}}
\newcommand{\na}{\nabla}
\newcommand{\calM}{\mathcal{M}}
\newcommand{\ubar}{\overline{u}}
\newcommand{\wbar}{\overline{w}}
\newcommand{\uu}{\underline{u}}
\newcommand{\norm}[2][]{\|#2\|_{#1}}
\newcommand{\nn}{\nonumber}
\newcommand{\kl}[1]{\left(#1\right)}

\newcommand{\ue}{u_{ε}}
\newcommand{\uet}{u_{εt}}
\newcommand{\uert}{u_{εrt}}
\newcommand{\uer}{u_{εr}}
\newcommand{\uerr}{u_{εrr}}
\newcommand{\uerrr}{u_{εrrr}}

\newcommand{\intni}{\int_0^{∞}}
\newcommand{\iBR}{\int_{B_R}}

\DeclareMathOperator{\supp}{supp}

\newtheorem{thm}{Theorem}
\newtheorem{lemma}[thm]{Lemma}
\theoremstyle{definition}
\newtheorem{remark}[thm]{Remark}
\newtheorem{definition}[thm]{Definition}

\usepackage{xcolor}

\usepackage{authblk}

\author{Marek Fila\footnote{e-mail: fila@fmph.uniba.sk} }
\author{Johannes Lankeit\footnote{e-mail: jlankeit@math.upb.de}}

\affil{\footnotesize Department of Applied Mathematics and Statistics, Comenius University,\\ Mlynsk\'a dolina, 84248 Bratislava, Slovakia}

\title{Solutions with a standing singularity of the gradient for a semilinear parabolic equation}
\title{A semilinear parabolic equation with solutions exhibiting a standing gradient singularity}
\title{Standing singularity of the gradient\\ in a semilinear parabolic equation}
\title{Lack of smoothing for bounded solutions\\ of a semilinear parabolic equation}

\begin{document}
\maketitle 

\begin{abstract}
\noindent
\textbf{Abstract.} 
We study a semilinear parabolic equation that possesses global bounded weak solutions whose gradient has a singularity in the interior
of the domain for all $t>0$. The singularity of these solutions is of the same type as the singularity of a stationary solution to which they
converge as $t\to\infty$.
\end{abstract}

For a bounded, smooth domain $\Omega\subsetℝ^n$, $T>0$ and $A\in ℝ$, consider solutions of the problem 
 \begin{equation}\label{eq:g}
 \begin{cases} u_t = Δu + g(u,\nabla u) & \text{in } \Om\times(0,T),\\
  u=A & \text{on } ∂\Om\times(0,T),\\
  u(\cdot,0)=u_0 & \text{in } \Ombar.
 \end{cases}
 \end{equation}
It is well known (see \cite[Thm.~VI.4.2]{LSU}) that this problem has a unique classical solution for small $T>0$ provided $g\in C^1(ℝ^{n+1})$, $u_0\in C^1(\Ombar)$ and $u_0=A$ on $∂\Om$. In this paper we study a particular case of problem~\eqref{eq:g} in a radially symmetric setting in $B_R:=\{x\in ℝ^n~\mid~|x|<R\}$, $R>0$, where $g$ is a smooth function of $u$ and $u_r$ but $u_0$ is only H\"older continuous in $\Bbar_R$, and there is no classical solution for any $T>0$. In our example, the global bounded weak solution emanating from $u_0$  maintains the singularity of the gradient of $u_0$ for all $t>0$. Thus, there is no smoothing effect which one usually expects from a semilinear uniformly parabolic equation.
 
The equation we will be  interested in is the following:
 \begin{equation}\label{eq:u}
  u_t = Δu + u u_r^3 \qquad \text{in}\quad (B_R\setminus\set{0})\times (0,\infty).
  \end{equation}
%

For $n\in ℕ$, $n\ge 2$, the function 
\begin{equation}\label{eqdef:alpha-ustar}
 u^*(r)= - αr^{\f13} \quad\text{for}\quad r>0, \qquad \text{where } α:=\sqrt[3]{9n-15},
\end{equation}
forms a stationary solution of \eqref{eq:u} (for any $R>0$ both in $B_R\setminus\set{0}$, cf. Lemma~\ref{lem:stationary}, and -- in the weak sense -- in $B_R$, see Lemma~\ref{lem:stationary-weak}).

We will impose several conditions on the initial data $u_0$ (and refer to \eqref{u0cond} in Section~\ref{sec:stationary-and-initdata} below for details) that, besides radial symmetry, essentially require that $u_0$ lies below the stationary solution, but is 'close' to it in a suitable sense. Under these conditions we will be able to show the global existence of solutions that retain the singularity in their gradient throughout the evolution. 
 
\begin{thm}\label{thm:classical}
 Let $n\ge2$ and $0<R<\sqrt{\f38(3n-5)(2n-3)^3}$. Assume that $u_0$ satisfies \eqref{u0cond}. Then there is a function 
 \begin{equation}\label{eq:solutionregularity}
  u\in C(\Bbar_R\times[0,∞))\cap C^{2,1}((B_R\setminus\set0)\times(0,∞)) 
 \end{equation}
 which solves 
  \begin{equation}\label{ueq}
 \begin{cases}
  u_t=Δu+uu_r^3 \qquad &\text{in } (B_R\setminus\set0)\times(0,∞),\\
  u(0,t)=0, \quad u(R,t)=u^*(R) &\text{for all } t>0,\\
  u(\cdot,0)=u_0 &\text{in } \Bbar_R,
  \end{cases}
 \end{equation}
 in the classical sense. 
This solution is unique in the class of
functions satisfying 
$u_r\le0$ in $(B_R\setminus\set0)\times(0,∞)$ and \eqref{eq:solutionregularity}. Moreover, it holds that
 \begin{equation}\label{eq:gradientsingularity}
   \lim_{r\searrow 0} u_r(r,t)=-\infty \qquad \text{for every } t>0.
 \end{equation}
\end{thm}

For a more precise description of the singularity see
Remark~\ref{rem:form-of-the-singularity-remains}. Next we show that the function $u$ from Theorem~\ref{thm:classical} solves the equation from \eqref{eq:u} also in $B_R\times(0,∞)$ in a suitable weak sense.

\begin{thm}\label{thm:weak}
 In addition to the assumptions of Theorem~\ref{thm:classical} let $n\ge 3$. Then the solution $u$ from Theorem~\ref{thm:classical} is a weak solution of
 \begin{equation}\label{ueq-weak}
 \begin{cases} u_t = Δu + uu_r^3 & \text{in } B_R\times(0,∞),\\
  u=u^*(R) & \text{on } ∂B_R\times(0,∞),\\
  u(\cdot,0)=u_0 & \text{in } \Bbar_R.
 \end{cases}
 \end{equation}
By this we mean that
 \begin{equation}\label{eq:integrabilityur}
  uu_r^3 \in L^1_{loc}(B_R\times[0,∞)) \quad\text{ and }\quad \na u\in L^1_{loc}(B_R\times[0,∞)),
 \end{equation}
 and for every $φ\in C_c^\infty(B_R\times(0,∞))$ we have 
 \begin{equation}\label{eq:weak}
  -\intni\iBR φ_tu = -\intni\iBR \na u\cdot\na φ + \intni\iBR uu_r^3φ.
 \end{equation}
 We note that Theorem~\ref{thm:classical} guarantees that the initial and boundary conditions are satisfied.
\end{thm}

Concerning the large-time behavior we establish the following:
\begin{thm}\label{thm:convergence}
 Under the assumptions of Theorem~\ref{thm:classical}, 
 \[
  u(\cdot,t)\to u^* \quad \text{as } t\to \infty. 
 \]
 This convergence is uniform in $B_R$ and occurs with an exponential rate.
\end{thm}

An equation closely related to \eqref{eq:u} has been studied before in \cite{AF, FL}, see also \cite{QS}. It was shown in \cite{AF} that interior gradient blow-up may occur for solutions of the problem
\[
 \begin{cases} u_t = u_{xx} + f(u)|u_x|^{m-1}u_x \, , \qquad x\in (-1,1),\\
  u(\pm 1,t)=A_\pm \, ,
 \end{cases}
\]
where $m>2$ and $f(u)=u$, for example. A global continuation after the interior gradient blow-up has been constructed recently in \cite{FL} for $m=3$.

For various parabolic equations, solutions with a standing or moving singularity have been investigated by many authors. We shall give some references below. But in these references it is the solution itself that is unbounded while in the present work only the gradient stays unbounded.

For the equation
\[
u_t=\nabla\cdot(u^{m-1}\nabla u),
\]
solutions with standing singularities were considered in \cite{C,CV1,CV2, HK, QV, VW} for various ranges of $m>0$, $m\not=1$, and some results on moving singularities for the same equation can be found in \cite{FTY}.

Results on moving singularities for the heat equation were established in \cite{KT1, TY1} and for semilinear equations of the form
\[
u_t=\Delta u \pm u^p,\qquad p>1,
\]
in  \cite{KT2, KT3, SY1,SY2,SY3, TY2}.

Next we describe the plan of the paper.
Due to the gradient singularity that the solutions 
have at the spatial origin, the notion of classical solvability is restricted to $(B_R\setminus\set0)\times(0,∞)$. In Section~\ref{sec:weak-and-classical} we therefore begin by establishing a connection between classical solutions in $(B_R\setminus\set0)\times(0,∞)$ and weak solutions in $B_R\times(0,∞)$. 

Section~\ref{sec:stationary-and-initdata} will be concerned with the stationary solution $u^*$ mentioned in \eqref{eqdef:alpha-ustar} (and already use the result of Section~\ref{sec:weak-and-classical}). At the end of this section, we give a precise formulation of the conditions on $u_0$ that the theorems require (and that involve the stationary solution). 

We will construct the solutions between a super- and a subsolution. As a supersolution we will use $u^*$, finding the subsolution will be the goal of Section~\ref{sec:subsolution}. To this aim, we will find a solution $v$ to a (formal) linearization of \eqref{ueq} (see 
Lemma~\ref{lem:linearized}) and then ensure that $u^*-v$ is a subsolution (Lemma~\ref{lem:subsolution}). (This is also the source of the restriction on $R$ in the theorems.)

The actual construction of solutions takes place in Section~\ref{sec:existence}. We first restrict the spatial domain to $\Ome:=B_R\setminus B_{ε}$, for the choice of the boundary value on the new boundary $∂B_{ε}\times(0,∞)$ already relying on $u^*-v$ from Section~\ref{sec:subsolution}. In Section~\ref{subsec:approx-system}, we take care of the solvability of this problem. (Classical existence results become applicable after replacing the nonlinearity $u_r^3$ by $f(u_r)$, see Lemma~\ref{lem:ex:ue}, and until Lemma~\ref{lem:ue-without-f}, we will have derived sufficient estimates  allowing for removal of $f$, though still $ε$-dependent.) Section~\ref{subsec:apriori} will then be concerned with $ε$-independent estimates in preparation of a compactness argument leading to the existence of solutions. The key to this part will lie in a comparison principle applied to high powers of $u_r$ (see Lemma~\ref{lem:bound:uer:Lploc-prelim}). 
This is a modification of a classical technique which involves $|\nabla u|^2$ and originated in \cite{Bernstein}. 
Section~\ref{subsec:solve} will contain the passage to the limit $ε\searrow 0$ (Lemma~\ref{lem:ex}) and deal with \eqref{ueq} and \eqref{ueq-weak}. 

In Section~\ref{sec:theoremproofs}, finally, we give the proofs of the theorems. By this time, they will only consist in collecting the right lemmata previously proven, and will be accordingly short. 

\section{Relation between classical and weak solutions}\label{sec:weak-and-classical}
Of course, every classical solution of \eqref{ueq} is also a weak solution of \eqref{ueq} -- in $(B_R\setminus\set0)\times(0,∞)$, which means that the singularity appears on the boundary of the domain. 
In order to interpret classical solutions in $(B_R\setminus\set0)\times(0,∞)$ as weak solutions in $B_R\times(0,∞)$, we merely require suitable integrability properties of the derivative near $0$: 

\begin{lemma}\label{lem:isweaksol} Let $n\ge 1$ and $R>0$. 
 Assume that a radially symmetric function 
 \[
  u\in C(\Bbar_R\times[0,∞)) \cap C^{2,1}((B_R\setminus\set{0})\times(0,∞)) 
 \]
satisfies \eqref{eq:integrabilityur},
\eqref{ueq}, and for every $T>0$ we have that
 \begin{equation}\label{eq:integralassumption}
  \lim_{ε\to 0} \f1{ε} \int_0^T \int_0^{ε} r^{n-1} |u_r(r,t)| dr dt= 0. 
 \end{equation}
 Then \eqref{eq:weak} holds for every $φ\in C_c^\infty(B_R\times(0,∞))$.
\end{lemma}
\begin{proof}
 For every $ψ\in C_c^\infty((B_R\setminus\set0)\times(0,∞))$ we obtain
 \[
  -\intni\iBR ψ_tu = -\intni\iBR \na u\cdot\na ψ + \intni\iBR uu_r^3ψ,
 \]
 as $u$ solves the equation classically in $(B_R\setminus\set0)\times(0,∞)$. 
 
 We introduce a non-decreasing cut-off function $χ\in C^\infty(ℝ)$ with $0\le χ'\le 2$ and $χ(0)=0$, $χ\equiv 1$ on $[1,∞)$ and let $χ_{ε}(x):=χ(\f{|x|}{ε})$. 
 
 We let $φ\in C_c^\infty(B_R\times(0,∞))$ and note that for every positive $ε$, $ψ:=χ_{ε}φ$ belongs to $C_c^\infty((B_R\setminus\set0)\times(0,∞))$. 
 
 \begin{align*}
  -\intni\iBR φ_t u =& -\intni\iBR ψ_t u -\intni\iBR φ_t(1-χ_{ε})u \\
  =& -\intni\iBR \na u\cdot\na ψ + \intni\iBR uu_r^3ψ-\intni\iBR φ_t(1-χ_{ε})u\\
  =& -\intni\iBR χ_{ε}\na u\cdot\na φ - \intni\iBR φ\na u\cdot \na χ_{ε} \\&+\intni\iBR uu_r^3φχ_{ε} -\intni\iBR φ_t(1-χ_{ε})u
 \end{align*}
 for every $ε>0$. As $χ_{ε}\to 1$ a.e. in $\supp φ$ and by \eqref{eq:integrabilityur} and  boundedness of $u$ each of the functions $\na u\cdot \na φ$, $uu_r^3φ$, $φ_tu$ belongs to $L^1(\supp φ)$, 
 \begin{align*}
  -\intni&\iBR χ_{ε}\na u\cdot\na φ +\intni\iBR uu_r^3φχ_{ε} -\intni\iBR φ_t(1-χ_{ε})u\\
  &\to -\intni\iBR \na u\cdot\na φ +\intni\iBR uu_r^3φ \qquad \text{as } ε\to 0
 \end{align*}
 by Lebesgue's dominated convergence theorem. 
 
 Moreover, $|\na χ_{ε}(x)|=|χ_{εr}(r)|=\f1{ε}χ'(\f{r}{ε})\le \f{2}{ε}$ if $r=|x|<ε$ and $|\na χ_{ε}(x)|=0$ if $|x|>ε$. With $T>0$ such that $\supp φ\subset B_R\times(0,T)$, we have
 \begin{align*}
  \left\lvert \intni\!\iBR φ\na u\cdot \na χ_{ε}\right\rvert &\le \norm[∞]{φ} \int_0^T\iBR |\na u||\na χ_{ε}|
\le \f{2}{ε} \norm[∞]{φ} \int_0^T\int_0^{ε} r^{n-1} |u_r(r,t)|  dr dt, 
 \end{align*}
 which vanishes as $ε\to 0$ according to \eqref{eq:integralassumption},
 and \eqref{eq:weak} follows.
\end{proof}

\section{The stationary solution and conditions on the initial data}\label{sec:stationary-and-initdata}

In \eqref{eqdef:alpha-ustar}, we have introduced a stationary solution $u^*$ to \eqref{eq:u}. In this section we first prove that the function from \eqref{eqdef:alpha-ustar} actually has this property (see Lemma~\ref{lem:stationary} for the classical, Lemma~\ref{lem:stationary-weak} for the weak sense) and then formulate the conditions on the initial data, which involve relations with $u^*$ and whose formulation we therefore had postponed.  


 
 \begin{lemma}\label{lem:stationary}
  Let $n\ge 2$. Then the function $u^*$ from \eqref{eqdef:alpha-ustar} 
  solves 
  \[
   Δu^* + u^*(u_r^*)^3=0 \qquad \text{in } ℝ^n\setminus\set{0}. 
  \]
 \end{lemma}
 \begin{proof}
  We use radial symmetry and the explicit form of $u^*$ to write 
  \begin{align*}
   Δu^* + u^*(u_r^*)^3 &= r^{1-n}\kl{r^{n-1} \kl{-\f{α}3 r^{-\f23}}}_r + αr^{\f13} \kl{\f{α}3 r^{-\f23}}^3 \\
   &= \f{α}{27}r^{-\f53}\kl{15-9n + α^3} =0. \qedhere
  \end{align*}
 \end{proof}

\begin{lemma}\label{lem:stationary-weak}
 Let $n\ge 2$. Then for any $R>0$ the function $u^*$ defined 
in \eqref{eqdef:alpha-ustar} is a weak solution of \eqref{ueq-weak}.
\end{lemma}
\begin{proof}
 In order to apply Lemma~\ref{lem:isweaksol}, we only have to check integrability of 
 $u^*(r)(u^*_r)^3(r)=\f{α^4}{27}r^{\f13-2}$ and $u^*_r(r)=\f{α}3 r^{-\f23}$, which is satisfied, and 
\[  \lim_{ε\to 0} \f1{ε}\int_0^{ε}r^{n-1}|u^*_r(r)|dr=\lim_{ε\to 0} \f{α}{3n-2} ε^{n-\f53} =0.
  \qedhere
\]
\end{proof}


Now and in the following, given any $n\in ℕ$ we let 
\begin{equation}\label{eqdef:nu}
 ν:=ν(n):=\f16\sqrt{36n^2-96n+61}.
\end{equation}

Having introduced $u^*$ and $ν$, we are now in a position to give the conditions on initial data that Theorems \ref{thm:classical}, \ref{thm:weak} and \ref{thm:convergence} have posed. 
\begin{subequations}\label{u0cond}
 \begin{align}
  &u_0\in C^2(\Bbar_R\setminus\set{0})\label{u0:C2},\\
  &u_0 \text{ is radially symmetric} \label{u0:radsym},\\
  &u^*\ge u_0, \label{u0:leq}\\
 &\limsup_{r\searrow 0} |r^{\f32-n-ν} (u^*(r)-u_0(r))|<\infty \label{u0:closetozero},\\
 &u_0(R)=u^*(R)\label{u0:aeussererrand},\\
 &\text{there is } C>0 \text{ such that } 0\ge u_{0r}(r)\ge -Cr^{-\f23}\quad \text{for every } r\in(0,R). \label{u0:derivativerestriction}
 \end{align}
\end{subequations}

\begin{remark}\label{rem:form-of-the-singularity-remains}
 The shape of the solution from Theorem~\ref{thm:classical} near the singularity of its gradient can be described more precisely than in \eqref{eq:gradientsingularity} by saying that 
 \eqref{u0:closetozero} continues to hold for $t>0$ in the sense that 
  \begin{equation*}
   \limsup_{r\searrow 0} |r^{\f32-n-ν} (u^*(r)-u(r,t))|<\infty \qquad \text{for all }t>0.
  \end{equation*}
  We will include a proof in the proof of Theorem~\ref{thm:classical} in Section~\ref{sec:theoremproofs}.
\end{remark}

 \section{Finding a subsolution}\label{sec:subsolution}

In order to construct a subsolution of \eqref{ueq} near $u^*$, we first find a solution of the (formal) linearization of \eqref{ueq} around $u^*$. 

\begin{lemma}\label{lem:linearized}
 Let $n\ge 2$, $C>0$, $ν$ as in \eqref{eqdef:nu}, $λ>0$. 
 Then the function 
 \begin{equation}\label{def:v}
  v(r,t) := C e^{-λ^2t} r^{n-\f32} J_{ν}(λr), \qquad r>0, t>0, 
 \end{equation}
 where $J_{ν}$ denotes the Bessel function of the first kind of order $ν$, solves 
 \[
  v_t = Δv + 3u^*u^{*2}_r v_r + u^{*3}_rv \qquad \text{in } (ℝ^n\setminus\set{0})\times(0,∞) 
 \]
 with $u^*$ taken from \eqref{eqdef:alpha-ustar}. 
\end{lemma}
\begin{proof}
 Let us recall that the function defined by $χ(r):=J_{ν}(λr)$, $r>0$, satisfies 
 \begin{equation}\label{bessel-basic}
  r^2χ''(r)+rχ'(r) + \left(λ^2r^2-ν^2\right)χ = 0 \qquad \text{for every } r>0. 
 \end{equation}
 We abbreviate $A:=4-2n$ and $B:=\f{3n-5}9$ and $δ:=n-\f32$ and note that 
 \begin{equation}\label{A-und-delta}
  2δ+A=1
 \end{equation}
 and 
 \begin{equation}\label{A-delta-B-und-nyquadrat}
  δ(δ-1)+Aδ-B =
-n^2+\f{8n}3-\f{61}{36} = -ν^2, 
 \end{equation}
 so that \eqref{bessel-basic}, \eqref{A-und-delta} and \eqref{A-delta-B-und-nyquadrat} for  
 $ψ(r):=r^{δ}χ(r)$, $r>0$, entail  
 \begin{align*}
  r^2 ψ''(r)&+A r ψ'(r) +λ^2r^2ψ(r) - B ψ(r)
= r^2(r^{δ}χ)'' + Ar (r^{δ}χ)' - B r^{δ}χ \\
  &=r^2δ(δ-1)r^{δ-2}χ + 2r^2δr^{δ-1}χ'+ r^2r^{δ}χ'' 
  + Aδr^{δ} χ + Ar^{δ+1} χ' - Br^{δ}χ \\
  &= r^{δ}\left(r^2χ'' + (2δ+A) r χ' + (δ(δ-1)+Aδ -B )χ\right)\\
  &= r^{δ}\left(r^2χ'' + rχ' - ν^2 χ \right) 
= - r^{δ} r^2λ^2 χ 
= -r^2λ^2 ψ \qquad \text{for } r>0,
 \end{align*}
 and 
 \[
  v(r,t)=Ce^{-λ^2t}r^{n-\f32}J_{ν}(λr)=Ce^{-λ^2t}ψ(r), \qquad r>0,\, t>0,
 \]
 solves 
 \begin{align*}
  v_t&= Ce^{-λ^2t}(-λ^2ψ)
= Ce^{-λ^2t}\left(ψ''+\f{A}rψ'-\f B{r^2}ψ\right)\\
  &= v_{rr} +\f{4-2n}r v_r + \f{5-3n}9 v
= v_{rr} + \f{n-1}r v_r -\f{α^3}{3r} v_r - \f{α^3}{27r^2} v\\
  &= Δv + 3u^*u^{*2}_r v_r + u^{*3}_rv \qquad \text{in } (ℝ^n\setminus\set{0})\times(0,∞),   
 \end{align*}
 where we have used that $α=\sqrt[3]{9n-15}$ and $u^*(r)=-αr^{\f13}$.
\end{proof}

\begin{definition}\label{def:x0x1}
 With $ν$ from \eqref{eqdef:nu}, we let $x_0>0$ and $x_1\in(0,x_0)$ be the first positive roots of the Bessel function $J_{ν}$ of the first kind and its derivative $J_{ν}'$, respectively. (As $ν>0$, $J_{ν}$ and $J_{ν}'$ are positive on $(0,x_0)$ and $(0,x_1)$, respectively.) 
\end{definition}

\begin{lemma}\label{lem:subsolution}
 Let $n\ge 2$, $C>0$, $λ>0$ and, with $x_1$ from Definition \ref{def:x0x1}, 
 \begin{equation}\label{def:R}
   0<R<\min\set{\f{x_1}{λ},\sqrt{\f38(3n-5)(2n-3)^3}}.
 \end{equation}
 With $v$ from \eqref{def:v}, the function 
 \[
  u := u^* - v
 \]
 then satisfies 
 \begin{equation}\label{unterloesung}
  u_t \le Δu + uu_r^3 \qquad \text{in } (B_R\setminus\set{0})\times(0,∞).
 \end{equation}
\end{lemma}
\begin{proof}
For $u=u^*-v$, we have 
\begin{align*}
 - uu_r^3&=- (u^*-v)(u^*-v)_r^3 = -u^*(u^*-v)_r^3 + v(u^*-v)_r^3\\
 &= -u^*(u^*_r)^3 + 3 u^*(u^*_r)^2v_r - 3 u^*u^*_r v_r^2 +u^*v_r^3\\
 &\quad +(u_r^*)^3v-3(u_r^*)^2vv_r + 3u_r^*vv_r^2 -vv_r^3\qquad \text{in } (B_R\setminus\set0)\times(0,∞).
\end{align*}

As $u^*$ is a stationary solution according to Lemma~\ref{lem:stationary} and by Lemma~\ref{lem:linearized} $v$ solves the linearized equation, we conclude   
\[
 u^*_t-Δu^*-u^*(u^*_r)^3 = 0 \qquad \text{in } (B_R\setminus\set0)\times(0,∞)
\]
and 
\[
 -v_t+Δv +3u^*(u_r^*)^2v_r + (u_r^*)^3v = 0\qquad \text{in } (B_R\setminus\set0)\times(0,∞). 
\]
Accordingly, in $(B_R\setminus\set0)\times(0,∞)$ we obtain 
\begin{align*}
 u_t - Δu -uu_r^3 &= u^*_t -v_t -Δu^*+Δv  
- u^*(u^*_r)^3 +3 u^*(u^*_r)^2v_r - 3 u^*u^*_r v_r^2 +u^*v_r^3\\
 &\quad +(u_r^*)^3v-3(u_r^*)^2vv_r + 3u_r^*vv_r^2 -vv_r^3\\
  &= - 3 u^*u^*_r v_r^2 -3(u_r^*)^2vv_r +u^*v_r^3  + 3u_r^*vv_r^2 -vv_r^3\\
%
%
%
  &= - α^2 C^2r^{-\f13} e^{-2λ^2t} (ψ'(r))^2 
+ αC^2 r^{-\f23} e^{-2λ^2t}ψ(r)ψ'(r)\\
  &\quad- α C^3 r^{\f13} e^{-3λ^2t} (ψ'(r))^3
- α C^3r^{-\f23} e^{-3λ^2t} ψ(r)(ψ'(r))^2 \\
  &\quad- C^4 e^{-4λ^2t} ψ(r)(ψ'(r))^3.
\end{align*}
 Due to $rλ\le Rλ\le x_1= \min\set{x_0,x_1}$, we have that 
\begin{equation}\label{eq:psiprime}
 ψ'(r)=\kl{n-\f32}r^{n-\f52}J_{ν}(λr)+λr^{n-\f32}J_{ν}'(λr) \ge 0 \qquad \text{for all } r\in(0,R), 
\end{equation}
 hence 
 \begin{align}
  u_t - Δu -uu_r^3 &\le - α^2 C^2r^{-\f13} e^{-2λ^2t} (ψ'(r))^2 
  + αC^2 r^{-\f23} e^{-2λ^2t}ψ(r)ψ'(r)\nn\\
  &= αC^2e^{-2λ^2t} ψ'(r)r^{-\f23} \left(-αr^{\f13} ψ'(r) +ψ(r)\right) \qquad\text{ in } (0,R).\label{eq:unterloesung-bew}
 \end{align}
From \eqref{eq:psiprime} and $λR<x_1=\min\set {x_0,x_1}$, we can also infer 
\[
 \f{ψ'(r)}{ψ(r)}=r^{-1} \left[(n-\f32)+\f{rλJ_{ν}'(rλ)}{J_{ν}(rλ)}\right]\ge \kl{n-\f32}r^{-1}\qquad \text{for all } r\in(0,R),   
\]
so that 
\[
-αr^{\f13} ψ'(r) +ψ(r)\le \kl{-αr^{-\f23}\kl{n-\f32}+1}ψ(r) \le \kl{-αR^{-\f23}\kl{n-\f32}+1}ψ(r)\le 0
\]
for every $r\in(0,R)$, 
because $R^{-\f23}\ge \kl{\f38(3n-5)(2n-3)^3}^{-\f13}=\kl{α\kl{n-\f32}}^{-1}$, 
hence \eqref{eq:unterloesung-bew} turns into \eqref{unterloesung}.
\end{proof}

\section{Existence}\label{sec:existence}

\subsection{An approximate problem} \label{subsec:approx-system}
Construction of the solution to \eqref{eq:u} will be based on an appropriately modified problem on $(B_R\setminus B_{ε})\times (0,∞)$. In preparation of suitable initial data, we first turn our attention to $u_0$. 

\begin{lemma}\label{choice-of-R-lambda-C-v}
Let $n\ge2$, $0<R<\sqrt{\f38(3n-5)(2n-3)^3}$. Assume that $u_0$ satisfies \eqref{u0cond}.  
Let $λ>0$ be such that $λR< x_1$. 
There is $C>0$ so that $v$ from \eqref{def:v} satisfies 
\begin{equation}\label{u0geq}
 u_0\ge u^*-v(\cdot,0) \qquad \text{in } B_R.
\end{equation}
\end{lemma}
\begin{proof}
 Since $λR<x_0$, known asymptotics of the Bessel function \cite[p. 360, (9.1.7)]{abramowitz_stegun} yields the existence of $c_1=c_1(λ)>0$ 
 such that $c_1r^{ν}\le J_{ν}(λr)$ 
 for every $r\in[0,R]$. Therefore, \eqref{u0:closetozero} implies that for some $c_2>0$ we obtain
 \[
  \f{|u_0(r)-u^*(r)|}{r^{n-\f32}J_{ν}(λr)} \le c_2 \qquad \text{for every } r\in(0,R). 
 \]
 If we let $C\ge c_2$, this coincides with \eqref{u0geq}.
\end{proof}
%

\begin{definition}\label{def:R-lambda-C-v}
Now and in all of the following, we let $n$, $C$, $R$, $λ$, $v$ be as in Lemma~\ref{lem:subsolution} and Lemma~\ref{choice-of-R-lambda-C-v}. 
\end{definition}

\begin{definition}\label{def:u0eps}
Let $ε>0$ and $u_0$ satisfy \eqref{u0cond}.
We denote $\Ome:=B_R\setminus B_{ε}$.
Moreover, let $u_{0ε}\in C^2(\Omebar)$ be radially symmetric and such that  
\begin{subequations}
\begin{align}\label{defu0eps:value-at-eps}
 &u_{0ε}(ε)=u^*(ε)-v(ε,0),\\
 &u_{0r} \le u_{0εr}\le 0, \label{defu0eps:derivative}\\
 &u^*\ge u_{0ε}\ge u^*-v(\cdot,0),\label{defu0eps:estimate}\\
 &u_{0ε} = u_0 \quad \text{on the set } \set{r\in (ε,R] \mid u_0(r)<u^*(ε)-v(ε,0)-ε}\label{defu0eps:equalityonset}.
\end{align}
\end{subequations}
\end{definition}

\begin{remark}
 For \eqref{defu0eps:estimate}, we rely on Lemma~\ref{choice-of-R-lambda-C-v}; that the other conditions can be fulfilled is more immediate from \eqref{u0cond}. 
\end{remark}

\begin{remark}\label{rem:u0eps-and-u0-coincide}
As $u^*(ε)-v(ε,0)-ε\to 0$ as $ε\to 0$, \eqref{defu0eps:equalityonset} ensures that for every $δ>0$ there is $ε_0>0$ such that for all $ε\in(0,ε_0)$ we have $u_{0ε}=u_0$ on $B_R\setminus B_{δ}$. 
\end{remark}

\begin{definition}
\label{def:ceps}
Let $ε>0$. First let us note that 
\begin{equation*}
c_v:=-e^{λ^2t}v(ε,t)
\end{equation*}
is positive and constant with respect to $t$ according to \eqref{def:v}. \\
We choose $c^*_{ε} >1$ large enough so as to satisfy
\begin{subequations}
\begin{align}
 c^*_{ε} &> \sup_{[ε,R]} |u^*_r|, \label{defceps:ustarr}\\
 c^*_{ε} &> \sup_{[ε,R]} |(u^*-v(\cdot,0))_r|, \label{defceps:ustarminusvr}\\
 c^*_{ε} &> \sup_{[ε,R]} |u_{0εr}|, \label{defceps:uner}\\
 c_v &+\f{n-1}{ε}c^*_{ε} + u^*(ε)(c^*_{ε})^3\le 0.  \label{defceps:uerinnen}
\end{align}
\end{subequations}
\end{definition}

\begin{definition}\label{def:feps}
We let $f_{ε}\in C_c^\infty(ℝ)$ be such that $f_{ε}(s)=s^3$ for every $s\in[-c^*_{ε},c^*_{ε}]$ (with $c^*_{ε}$ from Definition \ref{def:ceps}) and $f_{ε}\le 0$ on $(-∞,0)$. 
\end{definition}

With $u_{0ε}$ and $f_{ε}$ as in Definitions \ref{def:u0eps} and \ref{def:feps}, we now consider 
\begin{equation}\label{eq:ue:f}
 \begin{cases}
  \uet = Δ\ue + \ue f_{ε}(\uer) & \text{in } \Ome \times(0,∞),\\
  \ue(\cdot,t)|_{∂B_ε} = \kl{u^*-v(\cdot,t)}|_{∂B_ε} & \text{for all } t>0,\\
  \ue(\cdot,t)|_{∂B_R}=u_0(R)=u^*(R) & \text{for all } t>0,\\
  \ue(\cdot,0)=u_{0ε} & \text{in } \Omebar.
 \end{cases}
\end{equation}

By classical theory for parabolic PDEs, this problem has a solution. 

\begin{lemma}\label{lem:ex:ue}
 Let $ε>0$. Then \eqref{eq:ue:f} has a unique solution 
 \[
  \ue \in C^{β,\f{β}2}(\Omebar\times[0,\infty))\cap C^{2+β,1+\f{β}2}(\Ome\times(0,\infty))\quad \text{ with }\quad 
\na\ue \in L^\infty_{loc}(\Omebar\times[0,∞))
 \]
 for some $β\in(0,1)$. This solution is radially symmetric. 
\end{lemma}
\begin{proof}
 Boundedness of $f_{ε}$ and the regularity requirements on $u_{0ε}$ ensure applicability of \cite[Thm. V.6.2]{LSU}, which yields existence and uniqueness of the solution. Radial symmetry of $u_{0ε}$ together with the uniqueness assertion implies radial symmetry of the solution. 
\end{proof}

Later (in Lemmata \ref{lem:gradientbound} and \ref{lem:bound:uer:Lploc-prelim}) we want to invoke   comparison principles for the derivative. In order to make them applicable, we need slightly more regularity than provided by Lemma~\ref{lem:ex:ue}.

\begin{lemma}\label{lem:ue:higherregularity}
Let $ε>0$. Then there is $β\in(0,1)$ such that
 \[
  \ue\in C^{3+β,\f{3+β}2}(\Ome\times(0,\infty))
 \qquad\text{and} \qquad
  \na \ue \in C^{β,\f{β}2}(\Omebar \times[0,\infty)).
 \]
\end{lemma}
\begin{proof}
Letting $η\in C_c^\infty(\Ome\times(0,\infty))$ we observe that $ηu$ solves $(ηu)_t=Δ(ηu)+g$, where $g=-η_tu-2\naη\cdot\na u-uΔη + ηuf_{ε}(u_r)$ and that, thanks to $u\in C^{2+β,1+\f{β}2}(\supp η)$ by Lemma~\ref{lem:ex:ue}, $g\in C^{1+β,\f{1+β}2}(\Omebar\times(0,\infty))$. \cite[Thm. IV.5.2]{LSU} therefore implies $ηu\in C^{3+β,\f{3+β}2}(\Omebar\times[0,\infty))$. 
Hölder continuity of $\na \ue$ up to $t=0$ and to the spatial boundary follows from \cite[Thm. 4.6]{Lieberman_AnnScPisa_86}.
\end{proof}

As a first estimate of $\ue$, the following Lemma~not only affirms boundedness of $\ue$, but also forms the foundation of estimate \eqref{eq:uest} for $u$.

\begin{lemma}\label{lem:ue-bd}
 Let $ε>0$. Then 
 \begin{equation}\label{comparison:ue}
  u^* \ge \ue \ge u^*-v \qquad \text{ in  } \Ome\times (0,∞).
 \end{equation}
\end{lemma}
\begin{proof}
 Due to \eqref{defceps:ustarr} and \eqref{defceps:ustarminusvr}, each of the functions $w\in\set{u^*,\ue,u^*-v}$ satisfies $f_{ε}(w_r)=w_r^3$ in $\Ome\times(0,∞)$ and hence for $w\in\set{u^*,\ue}$ we have 
 \[
  w_t = Δw + f_{ε}(w_r)w, 
 \]
 whereas $w_t\le Δw+f_{ε}(w_r)w$ for $w=u^*-v$ (cf. Lemma~\ref{lem:subsolution}). 
 By construction, $u^*(R)=\ue(R,t)\ge u^*(R)-v(R,t)$ and $u^*(ε)\ge\ue(ε,t)=u^*(ε)-v(ε,t)$ for all $t>0$, and $u^*\ge u_{0ε}\ge u^*-v(\cdot,0)$, so that the comparison principle (\cite[Prop.~52.6]{QS}) implies \eqref{comparison:ue}. 
\end{proof}

We prepare for an estimate of $\uer$ by comparison, first 
providing some information on its value on the spatial boundary, beginning with the outer part $∂B_R\times(0,∞)$. 

\begin{lemma}\label{lem:uer:bounds:aussen}
 For every $ε>0$ and $t>0$ we have 
 \[
  u^*_r(R)\le \uer(R,t) \le 0.
 \]
\end{lemma}
\begin{proof}
 Since $u^*(R)=\ue(R,t)$ for all $t>0$, \eqref{comparison:ue} shows that $u^*_r(R)\le \uer(R,t)$ for all $t>0$. Moreover, $\uu(r,t):=u^*(R)$, $(r,t)\in[ε,R]\times[0,∞)$, satisfies $\uu_t\le Δ\uu +f(\uu_r)\uu$ in $(ε,R)\times(0,∞)$ and $\uu(R,t)\le \ue(R,t)$, $\uu(ε,t)\le \ue(ε,t)$ for all $t>0$ and $\uu(r,0)\le \ue(r,0)$ for all $r\in(ε,R)$. By the comparison principle \cite[Prop.~52.6]{QS} therefore $\ue(r,t)\ge u^*(R)=\ue(R,t)$ for every $(r,t)\in(0,R)\times(0,∞)$ so that $\uer(R,t)\le 0$ for every $t>0$.
\end{proof}

On the inner boundary, we first establish the sign of $\uer$. 
\begin{lemma}\label{lem:uer:nonpositive:innen}
 For every $ε>0$ and $t>0$ it holds that
 \[
  \uer(ε,t)\le 0.
 \]
\end{lemma}
\begin{proof}
 With $\calM[ϕ]:= ϕ_t - Δϕ - \ue\uer^2ϕ_r$ and $\ubar(x,t):=u^*(ε)-v(ε,t)$ for $(x,t)\in\Omebar\times[0,∞)$, we have 
 \[
  \calM[\ue]=0, \quad \calM[\ubar]=\ubar_t=-v_t(ε,t)\ge 0 \qquad \text{in } \Ome\times(0,∞), 
 \]
 which together with $\ue(ε,t)=\ubar(ε,t)$, $\ue(R,t)=u^*(R)\le \ubar(R,t)$ for all $t>0$ and the consequence $u_{0ε}(r)\le u_{0ε}(ε) = \ubar(r,0)$ of \eqref{defu0eps:derivative} and \eqref{defu0eps:value-at-eps} enables us to invoke \cite[Prop.~52.6]{QS} once more to conclude $\ue(r,t)\le \ubar(r,t)=\ue(ε,t)$ for all $r\in(ε,R)$ and $t>0$, which implies $\uer(ε,t)\le 0$ for all $t>0$.
\end{proof}

The upper estimates in Lemma~\ref{lem:uer:bounds:aussen} and Lemma~\ref{lem:uer:nonpositive:innen} determine the sign of $\uer$ throughout $\Ome\times[0,∞)$.
\begin{lemma}\label{lem:uer:nonpositive}
 Let $ε>0$. Then 
 \[
  \uer\le 0 \qquad \text{in } \Ome\times[0,∞).
 \]
\end{lemma}
\begin{proof}
 As $w:=\uer$ belongs to $C(\Omebar\times(0,∞))\cap C([0,∞);L^2(\Omebar))$ with $w_t,\nabla w,D^2w\in L^2_{loc}(\Ome\times(0,∞))$ by Lemma~\ref{lem:ue:higherregularity}, solves $w_t=Δw + f_{ε}(\uer)w + \ue f'_{ε}(\uer) w_r$ in $\Ome\times(0,∞)$, $f(\uer)$ is bounded in $\Ome\times(0,∞)$ due to boundedness of $f_{ε}$, and so is $\ue f'_{ε}(\uer)$ because of Lemma~\ref{lem:ue-bd}, we can apply \cite[Prop.~52.8]{QS} to conclude nonpositivity of $w$ from nonpositivity of $w$ on $\Ome\times\set{0}$ (see \eqref{defu0eps:derivative}) and on $∂\Ome\times(0,∞)$ as guaranteed by Lemmata \ref{lem:uer:bounds:aussen} and \ref{lem:uer:nonpositive:innen}.
\end{proof}

We now turn our attention to the counterpart of Lemma~\ref{lem:uer:nonpositive:innen}.

\begin{lemma}\label{lem:uer:bounded:innen}
 For every $ε>0$ we obtain
 \[
  \uer(ε,t)\ge -c^*_{ε} 
 \]
 for every $t\in(0,∞)$, where $c^*_{ε}$ is as in Definition \ref{def:ceps}. 
\end{lemma}
\begin{proof}
We define $\uu(r,t):=(u^*-v)(ε,t)+c^*_{ε}(ε-r)$. Then $\uu(ε,t)=\ue(ε,t)$ for all $t>0$ due to the boundary condition in \eqref{eq:ue:f}; 
by \eqref{defu0eps:value-at-eps} and \eqref{defceps:uner}, 
\[
 \uu(r,0)=u_{0ε}(ε)-c^*_{ε}(r-ε)\le u_{0ε}(ε) - \int_ε^r \sup |u_{0εr}| \le u_{0ε}(r), 
\]
for every $r\in(ε,R)$, and similarly by \eqref{defceps:ustarr}, 
\[
 \uu(R,t) = u^*(ε)-v(ε,t)-c^*_{ε}(R-ε)\le u^*(ε)-c^*_{ε}(R-ε) \le u^*(R)=\ue(R,t)
\]
for every $t>0$. 
Due to Definition~\ref{def:feps}, $f_{ε}(c^*_{ε})=(c^*_{ε})^3$ and hence, by Lemma~\ref{lem:ue-bd} and \eqref{defceps:uerinnen},
\[
 \uu_t-Δ\uu-\ue f_{ε}(\uu_r) = -v_t(ε,t)+\f{n-1}r c^*_{ε} + (c_{ε}^*)^3\ue \le e^{-λ^2t}c_v+\f{n-1}{ε}c^*_{ε}+u^*(ε)(c^*_{ε})^3\le 0.
\]
Therefore, comparison (\cite[Prop.~52.6]{QS}) implies 
\[
 \ue(r,t)\ge \uu(r,t) \qquad \text{for all } t>0,\, r\in(ε,R),
\]
and as $\ue(ε,t)=\uu(ε,t)$ for every $t>0$, this shows that $\uer(ε,t)\ge \uu_r(ε,t)=-c_{ε}^*$ for every $t>0$. 
\end{proof}

The previous lemmata and a first Bernstein-type comparison of $\uer^2$ confirm that including $f_{ε}$ in \eqref{eq:ue:f} -- although necessary for application of the classical existence theorems -- has not altered the equation. 

\begin{lemma}\label{lem:gradientbound}
 For every $ε>0$ we have
 \[
  \sup_{\Ome\times(0,∞)} |\na \ue|\le c_{ε}^*.
 \]
\end{lemma}
\begin{proof}
 We let $\calM[ϕ]:=ϕ_t-Δϕ-f'_{ε}(\uer)\ue ϕ_r$. 
Then $\calM[c^*_{ε}]=0$ and 
 \begin{align*}
  \calM[|\na \ue|^2] &= 2\na \ue\cdot\naΔ\ue + 2|\na \ue|^2f_{ε}(\uer) + 2\ue f'_{ε}(\uer)\na \ue\cdot \na \uer \\
  &\quad  - \nabla\cdot(2D^2\ue \na \ue) - 2 f'_{ε}(\uer)\ue\na \ue\cdot \na\uer\\
  &= 2|\na\ue|^2f_{ε}(\uer)-2|D^2\ue|^2\qquad  \text{in } \Ome\times(0,∞). 
 \end{align*}
 In view of Lemma~\ref{lem:uer:nonpositive}, $\calM[|\na\ue|^2]\le 0$. 
 Lemma~\ref{lem:uer:bounds:aussen} and \eqref{defceps:ustarr} together with Lemmata \ref{lem:uer:nonpositive:innen} and \ref{lem:uer:bounded:innen} show that $(c_{ε}^*)^2\ge |\na\ue|^2$ on $∂\Ome\times(0,∞)$, and \eqref{defceps:uner} ensures the same on $\Ome\times\set0$.
 Therefore, comparison (in the form of \cite[Prop.~52.10]{QS}, if one allows $f$ to also depend on $t$ there -- the necessary adaptations in the corresponding proof are minor) proves $\sup_{\Ome\times(0,∞)} |\na \ue|^2\le (c_{ε}^*)^2$ and thus the lemma. 
\end{proof}

\begin{lemma}\label{lem:ue-without-f}
 The function $\ue$ solves 
\begin{equation}\label{eq:ue}
 \begin{cases}
  \uet = Δ\ue + \ue \uer^3 & \text{in } \Ome \times(0,∞),\\
  \ue|_{∂B_ε}(\cdot,t) = u^*-v(\cdot,t)|_{∂B_ε}& \text{for all } t>0,\\
  \ue|_{∂B_R}(\cdot,t)=u_0(R)=u^*(R) & \text{for all } t>0,\\
  \ue(\cdot,0)=u_{0ε} & \text{in } \Omebar.
 \end{cases}
\end{equation}
\end{lemma}
\begin{proof}
 Lemma~\ref{lem:gradientbound} guarantees that $|\uer|=|\na\ue|\le c_{ε}^*$ in $\Ome\times(0,∞)$, therefore $f(\uer)=\uer^3$ by Definition~\ref{def:feps}, and Lemma~\ref{lem:ue-without-f} becomes a corollary of Lemma~\ref{lem:ex:ue}.
\end{proof}

\subsection{A priori estimates}\label{subsec:apriori}
Inspired by the reasoning in \cite[Sec. 2]{AF}, which goes back to \cite{Bernstein}, we will now obtain an $ε$-independent bound for $\uer$ from a comparison principle applied to, essentially, a large, even power of $\uer$. 
Lack of $ε$-independent control over $\uer$ on the inner boundary (for which we refer to Lemma~\ref{lem:uer:bounded:innen} and which is natural if seen in light of the unbounded derivative of $u^*$ near $r=0$) makes inclusion of a cutoff function necessary. 

\begin{lemma}\label{lem:bound:uer:Lploc-prelim}
 Let $p\ge 4$ be an even integer. 
 There is $c>0$ such that
 \begin{equation}\label{eq:weighted-uer-estimate}
  (r-δ)_+^{p+3} \uer^p(r,t) \le c(1+\sup_{r>δ} (r-δ)_+^{p+3} u_{0εr}^p+t)
 \end{equation}
 for every $δ>0$, $ε\in(0,δ)$ and $t>0$, $r\in(0,R)$.
\end{lemma}
\begin{proof}
 We define $c:=\max\set{1,R^{p+3}|u_r^*(R)|^p,3(p+3))^{p+3}|u^*(R)|^{p+3} + \kl{R\f{p(p+3)^2}{p-1}}^{\f{p+3}3}}$ and fix $δ>0$ and $ε\in(0,δ)$. 
 Letting $w(r,t):=(r-δ)_+^{p+3}\uer^p(r,t)$ for $(r,t)\in(δ,R)\times(0,∞)$, in $(δ,R)\times(0,∞)$ we compute 
 \[
  w_r= (p+3) (r-δ)_+^{p+2}\uer^p + p(r-δ)_+^{p+3}\uer^{p-1}\uerr
 \]
 and 
 \begin{align*}
  w_{rr} =& (p+2)(p+3) (r-δ)_+^{p+1} \uer^p + 2p(p+3)(r-δ)_+^{p+2}\uer^{p-1}\uerr\\
  &\quad + p(p-1)(r-δ)_+^{p+3}\uer^{p-2}\uerr^2+p(r-δ)_+^{p+3}\uer^{p-1}\uerrr
 \end{align*}
as well as 
 \[
  \uert = \uerrr + \f{n-1}{r} \uerr - \f{n-1}{r^2}\uer +\uer^4+3\ue\uer^2\uerr.
 \]

For  
$\calM[ϕ]:=ϕ_t-Δϕ-3\ue\uer^2ϕ_r$ we thus obtain from \eqref{eq:ue}  
 
 \begin{align*}
  \calM[w] &= p(r-δ)_+^{p+3} \uer^{p-1}\uert - w_{rr}-\f{n-1}{r}w_r - 3\ue\uer^2w_r\\
  &= p(r-δ)_+^{p+3} \uer^{p-1}\uerrr + p\f{n-1}{r}(r-δ)_+^{p+3}\uer^{p-1}\uerr \\
  &\quad -p\f{n-1}{r^2}(r-δ)_+^{p+3}\uer^p + p(r-δ)_+^{p+3}\uer^{p+3}+3p(r-δ)_+^{p+3}\ue\uer^{p+1}\uerr\\
  &\quad - (p+2)(p+3) (r-δ)_+^{p+1} \uer^p - 2p(p+3)(r-δ)_+^{p+2}\uer^{p-1}\uerr \\&
   \quad- p(p-1)(r-δ)_+^{p+3}\uer^{p-2}\uerr^2 - p(r-δ)_+^{p+3}\uer^{p-1}\uerrr \\
  &\quad -(p+3)\f{n-1}r (r-δ)_+^{p+2}\uer^p - p\f{n-1}r (r-δ)_+^{p+3}\uer^{p-1}\uerr \\
  & \quad -3(p+3)(r-δ)_+^{p+2} \ue\uer^{p+2}-3p(r-δ)_+^{p+3}\ue\uer^{p+1}\uerr \\
  &= - p\f{n-1}{r^2}(r-δ)_+^{p+3}\uer^p + p(r-δ)_+^{p+3}\uer^{p+3}\\
  &\quad -(p+3)(p+2)(r-δ)_+^{p+1}\uer^p - 2p(p+3)(r-δ)_+^{p+2}\uer^{p-1}\uerr \\
  &\quad - p(p-1)(r-δ)_+^{p+3}\uer^{p-2}\uerr^2  \\
  &\quad -(p+3)\f{n-1}r (r-δ)_+^{p+2}\uer^p - 3(p+3)(r-δ)_+^{p+2}\ue\uer^{p+2}\\
  & \le  p(r-δ)_+^{p+3}\uer^{p+3}  - 2p(p+3)(r-δ)_+^{p+2}\uer^{p-1}\uerr - p(p-1)(r-δ)_+^{p+3}\uer^{p-2}\uerr^2 \\
  &\quad - 3(p+3)(r-δ)_+^{p+2}\ue\uer^{p+2}\qquad \text{ in } (B_R\setminus B_{δ})\times (0,∞).
 \end{align*}
 Here, by Young's inequality 
 \begin{align*}
  -&2p(p+3)(r-δ)_+^{p+2}\uer^{p-1} \uerr\\ & \le p(p-1) (r-δ)_+^{p+3} \uer^{p-2} \uerr^2 + \f{p(p+3)^2}{p-1} (r-δ)^{p+1} \uer^p\\
  &\le p(p-1) (r-δ)_+^{p+3} \uer^{p-2} \uerr^2 + (r-δ)_+^{p+3}|\uer|^{p+3} +\left((r-δ)_+ \f{p(p+3)^2}{p-1}\right)^{\f{p+3}3}
 \end{align*}
and
 \begin{align*}
  -3(p+3)(r-δ)_+^{p+2}\ue\uer^{p+2} \le (r-δ)_+^{p+3}|\uer|^{p+3} + (3(p+3))^{p+3} |\ue|^{p+3} 
 \end{align*}
in $(B_R\setminus B_{δ})\times(0,∞)$. 
 Recalling the sign of $\uer$ from Lemma~\ref{lem:uer:nonpositive} and setting $c_1:= (3(p+3))^{p+3} |u^*(R)|^{p+3} + \left(R \f{p(p+3)^2}{p-1}\right)^{\f{p+3}3}$ we hence obtain 
 \[
  \calM[w]\le c_1 \qquad \text{in } (B_R\setminus B_{δ})\times(0,∞).
 \]
 Furthermore, 
 \[
  w(R,t)\le R^{p+3}\uer^p(R,t)\le R^{p+3} (u^*_r(R))^p=:c_2\qquad \text{for all } t>0
 \]
 by Lemma~\ref{lem:uer:bounds:aussen}. 
With $c=\max\set{c_1,c_2,1}$ and $\wbar:=c(1+\sup_{r>δ} (r-δ)_+^{p+3} u_{0εr}^p+t)$ we not only have $\calM[\wbar]=c\ge \calM[w]$ in $(B_R\setminus B_{δ})\times(0,∞)$, but also $\wbar(R,t)\ge c_2 \ge w(R,t)$ for all $t>0$ and $\wbar(r,0)\ge \sup_{r>δ} (r-δ)_+^{p+3} u_{0εr}^p \ge w(r,0)$ for all $r\in (0,R)$ as well as $\wbar(δ,t)\ge0 = w(δ,t)$ for all $t>0$. 
Comparison (again by means of an adaptation of \cite[Prop.~52.10]{QS}) allows us to conclude $(r-δ)_+^{p+3}\uer^p=w\le\wbar= c(1+\sup_{r>δ} (r-δ)_+^{p+3} u_{0εr}^p+t)$ in $(B_R\setminus B_{δ})\times(0,∞)$. Additionally, for $r\in(0,δ)$, the left-hand side of this inequality is zero, and \eqref{eq:weighted-uer-estimate} holds.
\end{proof}

Next we bring Lemma~\ref{lem:bound:uer:Lploc-prelim} in a more directly applicable form. 
\begin{lemma}\label{lem:estimate:uer:pointwise}
 Let $p\ge 4$ be an even integer. 
 For every $T>0$ there is $c>0$ such that 
 \[
  |\uer(r,t)| \le c r^{-\f{p+3}p}
 \]
 for every $ε>0$, $t\in[0,T]$, $r\in(2ε,R)$.
\end{lemma}
\begin{proof}
 Conditions \eqref{u0:derivativerestriction} and \eqref{defu0eps:derivative} ensure the existence of $c_1>0$ such that 
 \[
  |u_{0εr}|\le c_1r^{-\f23} \quad \text{on } B_R\setminus B_{ε}
 \]
 for every $ε>0$, and hence 
 \[
  (r-δ)_+^{p+3} u_{0εr}^p \le c_1 (r-δ)_+^{p+3} r^{-\f{2p}3}
 \]
 for every $r\in(δ,R)$ and $ε<δ$. Noting that $r\mapsto (r-δ)_+^{p+3}r^{-\f{2p}3}$ is increasing on $(δ,R)$ due to $p+3>\f{2p}3$, we conclude that 
 \[
  (r-δ)_+^{p+3} u_{0εr}^p \le c_1 R^{p-\f{2p}3} = c_1R^{\f{p}3}
 \]
 for every $r\in(δ,R)$ and $ε\in(0,δ)$. 
 Lemma~\ref{lem:bound:uer:Lploc-prelim} hence implies that 
  there is $c_2>0$ such that 
 \[
  (r-δ)_+^{p+3} \uer^p(r,t) \le c_2(1+t)
 \]
for every $δ>0$, $ε\in(0,δ)$ and $t>0$, $r\in(0,R)$.
 If we insert $r=2δ$, we obtain 
 \[
  |\uer(2δ,t)| \le c_3 (2δ)^{-\f{p+3}p} (1+t) 
 \]
 for every $δ>0$, $ε\in(0,δ)$, $t>0$, where $c_3:=2^{\f{p+3}p}c_2$. We conclude by letting $c:=c_3(1+T)$. 
\end{proof}

As preparation of the compactness argument that will finally establish existence of a solution of \eqref{ueq} in $(B_R\setminus\set0)\times(0,∞)$, we use classical regularity theory for parabolic PDEs and rely on Lemma~\ref{lem:estimate:uer:pointwise} as a starting point. 

\begin{lemma}\label{lem:hoelderboundC1}
 Let $β\in(0,1)$. 
 Let $K$ be a compact subset of $(B_R\setminus\set0)\times(0,∞)$. Then there are $ε_0>0$ and $c>0$ such that for every $ε\in(0,ε_0)$ 
 \[
  \norm[C^{1+β,\f{1+β}2}(K)]{\ue} \le c.
 \]
\end{lemma}
\begin{proof}
 Let us choose $δ>0$ so small that $(B_{δ}\times(0,∞))\cap K=\emptyset$.
 Let $η\in C_c^{∞}((B_R\setminus B_{δ})\times(0,∞))$ be such that $η\equiv 1$ on $K$. Then for each $ε\in(0,ε_0)$, $ε_0:=\f{δ}2$, $η\ue$ is well-defined on $(B_R\setminus B_{δ})\times(0,∞)$ and $(η\ue)(δ,t)=0$, $(η\ue)(R,t)=0$ for every $t>0$, $(η\ue)(r,0)=0$ for every $r\in(δ,R)$ and 
 \[
  (η\ue)_t= Δ(η\ue) + g_{ε} \qquad \text{in } (B_R\setminus B_{δ})\times(0,∞), 
 \]
 where $g_{ε}:=- \ue Δη - 2\na \ue \cdot \naη+η\ue \uer^3-η_t\ue$. Lemma~\ref{lem:estimate:uer:pointwise} enables us to find $c_1>0$ satisfying  
 \[
  \norm[L^{∞}((B_R\setminus B_{δ})\times(0,∞))]{g_{ε}}=\norm[L^\infty(\supp η)]{g_{ε}} \le c_1 
 \]
 for every $ε\in(0,ε_0)$. Consequently, \cite[Thm. 7.4, p.~191]{friedman_parabolic} shows that with some $c_2>0$, 
 \[
  \norm[C^{1+β,\f{1+β}2}((B_R\setminus B_{δ})\times(0,∞))]{η\ue}\le c_2
\qquad \text{for every } ε\in(0,ε_0). \qedhere
 \]
\end{proof}

Leveraging Lemma~\ref{lem:hoelderboundC1}, we can achieve higher regularity analogously. 
\begin{lemma}\label{lem:hoelderboundC2}
 Let $β\in(0,1)$. 
 Let $K$ be a compact subset of $(B_R\setminus\set0)\times(0,∞)$. Then there are $ε_0>0$ and $c>0$ such that 
 \[
  \norm[C^{2+β,1+\f{β}2}(K)]{\ue} \le c \qquad \text{for every }
ε\in(0,ε_0).
 \]
\end{lemma}
\begin{proof}
 Again, we choose $δ>0$ so small that $(B_{δ}\times(0,∞))\cap K=\emptyset$, $η\in C_c^{∞}((B_R\setminus B_{δ})\times(0,∞))$ such that $η\equiv 1$ on $K$ and $ε_0:=δ$ and consider the Dirichlet problem of $(η\ue)_t= Δ(η\ue) + g_{ε}$ in $(B_R\setminus B_{δ})\times(0,∞)$, with $g_{ε}:=- \ue Δη - 2\na \ue \cdot \naη+η\ue \uer^3-η_t\ue$. 
 Thanks to Lemma~\ref{lem:hoelderboundC1}, applied to the compact set $\supp η$, there is $c_1>0$ fulfilling 
 \[
  \norm[C^{β,\f{β}2}((B_R\setminus B_{δ})\times(0,∞))]{g_{ε}}=\norm[C^{β,\f{β}2}(\supp η)]{g_ε}\le c_1
\qquad \text{for every } ε\in(0,ε_0).
 \]
We can therefore rely on \cite[Thm.~3.6, p.~65]{friedman_parabolic} so as to conclude the existence of $c_2>0$ such that 
 \[
  \norm[C^{2+β,1+\f{β}2}((B_R\setminus B_{δ})\times(0,∞))]{η\ue}\le c_2
\qquad \text{for every } ε\in(0,ε_0).\qedhere
 \]
\end{proof}

In the next step we aim for lower Hölder regularity, but strive to include the boundaries at $r=R$ and $t=0$.
\begin{lemma}\label{lem:hoelderboundCalpha}
 There is $β\in(0,1)$ such that 
 for every compact subset $K$ of $(\Bbar_R\setminus\set0)\times[0,∞)$ there are $ε_0>0$ and $c>0$ satisfying 
 \[
  \norm[C^{β,\f{β}2}(K)]{\ue}\le c \qquad \text{for every } ε\in(0,ε_0). 
 \]
\end{lemma}
\begin{proof}
   We choose $δ>0$ so small that $(B_{δ}\times(0,∞))\cap K=\emptyset$ and let $ε_0\in(0,\f{δ}2)$ be such that $u_{0ε}=u_0$ on $B_R\setminus B_{δ}$ for every $ε\in(0,ε_0)$ (cf. Remark~\ref{rem:u0eps-and-u0-coincide}). With $η\in C_c^{∞}((B_R\setminus B_{δ})\times(0,∞))$ such that $η\equiv 1$ on $K$ and relying on Lemma~\ref{lem:estimate:uer:pointwise}, we can conclude from \cite[Thm. III.10.1]{LSU} that with some $c>0$, 
   \[
    \norm[C^{β,\f{β}2}(K)]{η\ue}\le c\qquad \text{for every } ε\in(0,ε_0), 
   \]
   where $β$ can be determined independently of $δ$, $K$ and $η$.
\end{proof}

\subsection{Solving the limit problem}\label{subsec:solve}
With these estimates at hand, we are ready to carry out the existence proof.

\begin{lemma}\label{lem:ex}
 There is a function $u\in C(\Bbar_R\times[0,∞))\cap
C^{2,1}((B_R\setminus\set0)\times(0,∞))$ solving \eqref{ueq}. 
 This function is radially symmetric, satisfies 
 \begin{equation}\label{eq:uest}
  u^*(r)\ge u(r,t)\ge u^*(r) - v(r,t) \qquad \text{for all } (r,t)\in [0,R]\times [0,∞)
 \end{equation}
 and, in particular, with some $c>0$ we have
 \begin{equation}\label{eq:uest-spezial}
  0\ge u\ge -cr^{\f13}\qquad \text{in } B_R\times [0,∞),
 \end{equation}
 as well as 
 \begin{equation}\label{eq:urle0}
  u_r\le 0 \qquad \text{in } (B_R\setminus\set0)\times (0,∞),
 \end{equation}
 and for every $T>0$ there is some $c=c(T)>0$ such that
 \begin{equation}\label{eq:urge}
  u_r>-cr^{-\f{31}{28}} \qquad \text{ in } (B_R\setminus\set0)\times (0,T).
 \end{equation}
\end{lemma}
\begin{proof}
 If we apply Lemmata \ref{lem:hoelderboundC2} and \ref{lem:hoelderboundCalpha} 
 to sequences of compact sets exhausting $(B_R\setminus\set0)\times(0,∞)$ and $(\Bbar_R\setminus\set{0})\times[0,∞)$, respectively, use the Arzel\`a-Ascoli theorem and a diagonalization procedure,  
 we obtain a sequence $(ε_j)_{j\in ℕ}\searrow 0$ and a function $u\in C((\Bbar_R\setminus\set{0})\times[0,∞))\cap C^{2,1}((B_R\setminus\set0)\times(0,∞))$ such that
 \begin{align}\label{ue-to-u}
  u_{ε_j}\to u \qquad &\text{ locally uniformly in } (\Bbar_R\setminus\set{0})\times[0,∞)\\
  &\text{ and with respect to the topology of }C^{2,1}((B_R\setminus\set0)\times(0,∞)). \label{ue-to-u-C2}
 \end{align}
The latter convergence statement \eqref{ue-to-u-C2} together with Lemma~\ref{lem:uer:nonpositive} already entails \eqref{eq:urle0}, whereas \eqref{eq:urge} similarly results from Lemma~\ref{lem:estimate:uer:pointwise} upon the choice of $p=28$. 

 Additionally, we define $u(0,t):=0$. Then $u$ is continuous in $\Bbar_R\times[0,∞)$. In light of \eqref{ue-to-u}, only continuity at $(0,t)$ for $t\ge 0$ remains to be proven. 
 Let $η>0$. Choose $δ>0$ such that $u^*(δ)-v(δ,0)>-η$. Then for every $ε\in (0,δ)$, every $r\in(0,δ)$ and every $t\ge 0$ we have $0\ge \ue(r,t)\ge \ue(δ,t)\ge u^*(δ)-v(δ,t)\ge u^*(δ)-v(δ,0)>-η$ and, by \eqref{ue-to-u}, hence $0\ge u(r,t)\ge -η$ for every $r\in(0,δ)$ and $t\ge 0$.
 
 Finally, \eqref{eq:uest} and hence \eqref{eq:uest-spezial} are obvious for $r=0$ and easily obtained from Lemma~\ref{lem:ue-bd} for $r>0$.
\end{proof}

\newcommand{\utilde}{\tilde{u}}
Theorem~\ref{thm:classical} also includes a uniqueness statement. The following lemma takes care of it. 
\begin{lemma}\label{lem:unique}
 Let $u$, $\utilde$ be functions satisfying 
 \begin{align*}
  &u,\utilde\in C^{2,1}((B_R\setminus\set0)\times(0,∞))\cap C(\Bbar_R\times[0,∞)),\\
  &\sup u_r \le 0,\quad \sup \utilde_r\le 0
 \end{align*}
 that solve \eqref{ueq}. Then $u=\utilde$. 
\end{lemma}
\begin{proof}
 The difference $w:=u-\utilde$ solves $w_t=Δw+bw_r+ cw$ in $(B_R\setminus\set0)\times(0,∞))$, where $b:=\utilde(u_r^2+u_r\utilde_r+\utilde_r^2)$ and $c:=u_r^3\le 0$. Moreover, $w=0$ on $(B_R\times\set{0}) \cup (∂(B_R\setminus\set{0})\times (0,∞))$, and \cite[Prop.~52.4]{QS} shows $w\le 0$. 
\end{proof}

The final piece of the proof of Theorem~\ref{thm:weak} is the combination of Lemma~\ref{lem:ex} with Lemma~\ref{lem:isweaksol}. 
\begin{lemma}\label{lem:have-found-a-weak-solution} 
 Let $n\ge 3$. 
 Then the function $u$ obtained in Lemma~\ref{lem:ex} is a weak solution of \eqref{ueq-weak}.
\end{lemma}
\begin{proof}
%
 We observe that according to \eqref{eq:urge} there is $c_1=c_1(T)$ such that 
 \[
  \f1{ε} \int_0^T \int_0^{ε} r^{n-1} |u_r(r,t)| dr dt \le \f{c_1T}{ε}\int_0^{ε} r^{n-1} r^{-\f{31}{28}} dr = \f{c_1T}{n-\f{31}{28}} ε^{n-\f{59}{28}}\to 0
 \]
 as $ε\to 0$. 
 By \eqref{eq:urge} and \eqref{eq:uest-spezial}
 \[
  |uu_r^3|\le c r^{\f13}r^{3\cdot(-\f{31}{28})}=c r^{\f{28-9\cdot31}{84}}= cr^{-\f{251}{84}}\qquad \text{in } B_R\times(0,∞),
 \]
 and because $-\f{251}{84}=-3+\f1{84} \ge -n$, hence $uu_r^3\in L^1_{loc}(B_R\times(0,∞))$. Finally, $|u_r|\le c r^{-\f{31}{28}} \in L^1_{loc}$ and Lemma~\ref{lem:isweaksol} becomes applicable. 
\end{proof}

\section{Proofs of the theorems}\label{sec:theoremproofs}


\begin{proof}[Proof of Theorem~\ref{thm:classical} and Remark~\ref{rem:form-of-the-singularity-remains}]
      Solvability is ensured by Lemma~\ref{lem:ex}, which by means of \eqref{eq:uest} also ensures that for every $t>0$ there are $c_1=c_1(t)>0$ and $c_2=c_2(t)>0$ such that 
      \[
       0\ge u^*(r)-u(r,t)\ge -v(r,t)\ge -c_1r^{n-\f32}J_{ν}(λr)\ge -c_2 r^{n-\f32+ν} \quad \text{for every } r\in [0,R].
      \]
      (The last estimate therein used $λR<x_0$ and \cite[p. 360, (9.1.7)]{abramowitz_stegun}.) 
      This proves Remark~\ref{rem:form-of-the-singularity-remains} and implies  \eqref{eq:gradientsingularity}.\\       
       Uniqueness of solutions, on the other hand, has been asserted in Lemma~\ref{lem:unique}.
     \end{proof}

\begin{proof}[Proof of Theorem~\ref{thm:weak}]
 This is the outcome of Lemma~\ref{lem:have-found-a-weak-solution}. 
\end{proof}

\begin{proof}[Proof of Theorem~\ref{thm:convergence}]
 The construction of $u$ during the proof of Theorem~\ref{thm:classical} had ensured that $u^*(r)\ge u(r,t)\ge u^*(r)-v(r,t)$ 
for all $(r,t)\in [0,R]\times[0,∞)$ (cf. \eqref{eq:uest}), and Theorem~\ref{thm:convergence} can be seen from the explicit definition \eqref{def:v} of $v$.
\end{proof}

\textbf{Acknowledgements.} The first author was supported in part by the Slovak
Research and Development Agency under the contract No. APVV-14-0378 and by the VEGA grant
1/0347/18. 
\def\cprime{$'$}

\end{document}